\documentclass{amsart}

\usepackage[utf8]{inputenc}
\usepackage[
textwidth=3cm,
textsize=small,
colorinlistoftodos]
{todonotes}
\usepackage[unicode=true, pdfusetitle,
 bookmarks=true,bookmarksnumbered=false,
 breaklinks=false,
 backref=false,
 colorlinks=true,
 linkcolor=blue,
 citecolor=blue,
 urlcolor=blue,
 final
]{hyperref}

\usepackage{bookmark}
\usepackage{enumerate}  

\usepackage[capitalise]{cleveref}

\usepackage{amssymb,amsmath,tikz-cd,dsfont}
\usepackage{tikz}
\usepackage{tikz-qtree}
\usepackage{comment}
\usetikzlibrary{arrows,matrix,positioning}
\usepackage[normalem]{ulem}

\usetikzlibrary{calc}
\usetikzlibrary{decorations.markings,decorations.pathmorphing}
\tikzset{/tikz/commutative diagrams/arrow style=tikz,>=stealth} 
\tikzset{mm/.style={execute at begin node=$\displaystyle, execute at end node=$}}
\tikzset{tikzob/.style={commutative diagrams/every diagram, every cell}}
\tikzset{tikzar/.style={commutative diagrams/.cd, every arrow, every label, font={\small}}}
\tikzset{tikzsquiggle/.style={decorate, decoration={
    snake,
    segment length=8pt,
    amplitude=.9pt,post=lineto,
    post length=2pt}}}
\tikzset{cross line/.style={preaction={draw=white, -, line width=6pt}}}

\newcommand{\newrefformat}[2]{}

\theoremstyle{plain}   
\newtheorem{thm}{Theorem}[section] 

\makeatletter\let\c@cor\c@thm\makeatother
\newtheorem{lemma}{Lemma}[section]
\makeatletter\let\c@lemma\c@thm\makeatother
\newtheorem{prop}{Proposition}[section]
\makeatletter\let\c@prop\c@thm\makeatother

\makeatletter\let\c@claim\c@thm\makeatother

\theoremstyle{definition}

\newtheorem{defn}{Definition}[section]
\makeatletter\let\c@defn\c@thm\makeatother

\makeatletter\let\c@const\c@thm\makeatother

\makeatletter\let\c@notn\c@thm\makeatother

\makeatletter\let\c@outline\c@thm\makeatother

\theoremstyle{remark}

\newtheorem{rem}{Remark}[section]
\makeatletter\let\c@rem\c@thm\makeatother
\newtheorem{ex}{Example}[section]
\makeatletter\let\c@ex\c@thm\makeatother

\makeatletter\let\c@observationn\c@thm\makeatother

\makeatletter
\let\c@equation\c@thm
\numberwithin{equation}{section}
\makeatother

\makeatletter
\let\c@table\c@thm
\numberwithin{table}{section}
\makeatother

\crefname{lemma}{Lemma}{Lemmas}
\crefname{thm}{Theorem}{Theorems}
\crefname{defn}{Definition}{Definitions}
\crefname{notn}{Notation}{Notations}
\crefname{const}{Construction}{Constructions}
\crefname{prop}{Proposition}{Propositions}
\crefname{rem}{Remark}{Remarks}
\crefname{cor}{Corollary}{Corollaries}
\crefname{equation}{Display}{Displays}
\crefname{ex}{Example}{Examples}

\newcommand{\be}[1]{\mathcal{P}_{#1}}
\newcommand{\subbe}[1]{\mathcal{Q}_{#1}}
\newcommand{\obe}[1]{P_{#1}}
\newcommand{\osubbe}[1]{Q_{#1}}
\newcommand{\tree}[1]{\mathcal{T}_{#1}}
\newcommand{\cantree}[1]{\mathcal{CT}_{#1}}
\newcommand{\beid}{\mathds{1}}

\DeclareMathOperator{\Id}{id}
\newcommand{\Set}{\mathbf{Set}}
\newcommand{\Cat}{\mathbf{Cat}}

\newcommand{\GSet}{G\mathbf{Set}}
\newcommand{\GCat}{G\mathbf{Cat}}
\newcommand{\GTop}{G\mathbf{Top}}
\newcommand{\NN}{\mathbb{N}}
\newcommand{\cC}{\mathcal{C}}
\newcommand{\nN}{\mathcal{N}}
\newcommand{\oO}{\mathcal{O}}
\newcommand{\pP}{\mathcal{P}}
\newcommand{\vV}{\mathcal{V}}
\newcommand{\wW}{\mathcal{W}}
\newcommand{\Si}{\Sigma}
\newcommand{\ga}{\gamma}
\newcommand{\si}{\sigma}
\newcommand{\ta}{\tau}
\newcommand{\cn}{\colon}
\newcommand{\rtarr}{\longrightarrow}
\newcommand{\chao}[1]{\widetilde{#1}}
\newcommand{\assoc}{\mathbf{Assoc}}
\newcommand{\cbe}{\pP}

\newcommand{\free}{\mathbb{F}}
\newcommand{\frN}{\mathfrak{O}}
\newcommand{\dia}{\diamondsuit}
\newcommand{\prim}{\operatorname{Prim}}
\newcommand{\Anglearrow}[1]{\rotatebox{#1}{$\Rightarrow$}}
\newcommand{\id}{\mathrm{id}}

\title{Biased permutative equivariant categories}

\author[Bangs]{Kayleigh Bangs}
\address{Department of Mathematics, Reed College, Portland, OR 97202}
\email{kabangs@reed.edu}

\author[Binegar]{Skye Binegar}
\address{School of Mathematics, Georgia Institute of Technology, Atlanta, GA 30332}
\email{skye@gatech.edu}

\author[Kim]{Young Kim}
\address{Department of Mathematics, Reed College, Portland, OR 97202}
\email{yokim@reed.edu}

\author[Ormsby]{Kyle Ormsby}
\address{Department of Mathematics, Reed College, Portland, OR 97202}
\email{ormsbyk@reed.edu}

\author[Osorno]{Ang\'{e}lica M. Osorno}
\address{Department of Mathematics, Reed College, Portland, OR 97202}
\email{aosorno@reed.edu}

\author[Tamas-Parris]{David Tamas-Parris}
\address{Department of Mathematics, Reed College, Portland, OR 97202}
\email{datamaspa@reed.edu}

\author[Xu]{Livia Xu}
\address{Department of Mathematics, Reed College, Portland, OR 97202}
\email{xinxu@reed.edu}

\begin{document}

\begin{abstract}
For a finite group $G$, we introduce the complete suboperad $\subbe G$ of the categorical $G$-Barratt--Eccles operad $\be G$.  We prove that $\be G$ is not finitely generated, but $\subbe G$ is finitely generated and is a genuine $E_\infty$ $G$-operad (i.e., it is $N_\infty$ and includes all norms).  For $G$ cyclic of order $2$ or $3$, we determine presentations of the object operad of $\subbe G$ and conclude with a discussion of algebras over $\subbe G$, which we call biased permutative equivariant categories.
\end{abstract}

\maketitle

\section*{Introduction}

The classifying space functor from categories to topological spaces allows the construction of spaces with desired structure from categories with similar, but usually easier to handle, structure. This is especially true for symmetric monoidal categories (categories with a binary operation that is unital, associative, and commutative up to coherent natural isomorphisms), which give rise to infinite loop spaces. This was proven independently by Segal \cite{segal} and May \cite{may72geo, May1974Einfty}, the latter using the theory of operads. 

The particular operad of interest in \cite{May1974Einfty} is the categorical Barratt-Eccles operad $\cbe$. Its algebras are \emph{unbiased} permutative categories. On the one hand, a (biased) permutative category is a symmetric monoidal category that is strictly associative and unital. Its structure is specified by a finite amount of information: the unit object (0-ary operation), the monoidal product (2-ary operation), and the symmetry (2-ary morphism). This structure is subject to a finite number of axioms. On the other hand, an unbiased permutative category, defined as an algebra over $\cbe$, is given by a collection of $n$-ary operations for all $n\geq 0$, that are compatible with each other in a way encoded by the operad. One can easily check that one obtains a biased permutative category from an unbiased one by restricting the structure. A harder result, that relies on the coherence theorem for symmetric monoidal categories \cite{maclane}, is that every biased permutative category gives rise to an unbiased one, thus giving a one-to-one correspondence between the two kinds of structure.

One perspective on this correspondence is that the operad given by the objects of $\mathcal{P}$, thought of as an operad in $\Set$, is finitely presented. More precisely, this operad is generated by a 0-ary operation (encoding the unit) and by a 2-ary operation (encoding the monoidal product), and all other operations can be obtained from these two using the symmetric group actions and the operad composition. As such, this is all the structure one needs to specify to give an algebra over $P = \operatorname{Ob} \cbe$. Moreover, all morphisms in $\mathcal{P}$ are generated by a single morphism between the two 2-ary operations. The coherence theorem in this setting says that a few specific relations on this morphism generate all the relations present in $\mathcal{P}$.

The operad $\cbe$ is constructed such that its classifying space is an $E_\infty$ operad in spaces, and thus, the classifying space of a permutative category is an $E_\infty$ space, and hence, an infinite loop space upon group completion.
In \cite{GM}, Guillou and May construct an equivariant analogue of the categorical Barratt-Eccles operad for a finite group $G$. This operad, $\be{G}$, has the property that its classifying space is a genuine $E_\infty$ $G$-operad, and thus, its algebras give rise to genuine equivariant infinite loop spaces. Because of this, Guillou and May define permutative $G$-categories as algebras over $\be{G}$.

Following \cite{GMMO}, one may ask if there is a biased definition of permutative $G$-categories, as there is for permutative categories. One of the main results of this paper, \cref{cor:PG-not-fin-gen} is that in the strictest sense, the answer is no for nontrivial groups $G$. Indeed, we prove that the object part of $\be{G}$ is \emph{not} finitely generated, meaning that one needs to specify infinitely many operations to give an algebra over it. 

Using the work of Rubin \cite{RubinRealization, RubinThesis}, we construct for each finite group $G$ a suboperad $\subbe{G}$ of $\be{G}$ that is still $E_\infty$, yet is finitely generated.  The key insight from Rubin, which is inspired by the work on $N_\infty$ operads of Blumberg and Hill \cite{BH}, is that the full suboperad generated by a collection of norms will be $E_\infty$, as long as one includes all the norms for orbits as generators.

Finally, in \cref{thm:pres-c2,thm:pres-c3} we give explicit presentations for the operads $\subbe{G}$ in the cases where $G=C_2$ and $G=C_3$. Although the statements of the proofs look very similar, the proofs that the relations given are sufficient are strikingly different. We use these results together with Rubin's coherence theorem for normed symmetric monoidal categories \cite{RubinThesis} to give a biased definition of $\subbe{G}$-algebras.

\subsection*{Organization}
In \cref{sec:prelim}, we recall necessary preliminary notions regarding permutations, operads in general, and the categorical $G$-Barratt--Eccles operad $\be G$ and its operad of objects $\obe G$.  In \cref{sec:gen}, we prove that $\obe G$ is not finitely generated for nontrivial $G$ (\cref{cor:PG-not-fin-gen}).  In \cref{sec:biased}, we introduce the finitely generated $E_\infty$ $G$-operads $\subbe G$ and determine presentations of the operads of objects when $G=C_2$ or $C_3$.  Finally, in \cref{sec:morphisms}, we define the notion of a biased permutative $G$-category for $G=C_2$ or $C_3$ and prove that these are in one-to-one correspondence with $\subbe G$-algebras.

\subsection*{Acknowledgements}
The authors express their deep gratitude to Jonathan Rubin, who very generously explained the results of \cite{RubinThesis} in detail.  This research was supported by NSF grant DMS-1709302.

\section{Preliminaries}\label{sec:prelim}

\subsection{Permutations}

Let $\Sigma_n$ be the symmetric group on $n$ letters. Throughout the paper we denote elements in $
\Sigma_n$ using cycle notation. For $\sigma \in \Sigma_n$, let $M_{\sigma}$ denote the permutation matrix representing $\sigma$, that is,
\[M_{\sigma} = 
\begin{pmatrix} 
e_{\sigma(1)} & \cdots & e_{\sigma(n)} 
\end{pmatrix}.\]

For $\sigma \in \Sigma_n$, $k_1, ..., k_n \geq 0$, let $k=k_1+\dots + k_n$, and think of $\{k_1,\dots,k_n\}$ as a partition of $\{1,\dots,k\}$ into $n$ (possibly empty) blocks. We define the \emph{block permutation} $\sigma \langle k_1, \dots, k_n \rangle$ to be the permutation in $\Sigma_{k}$ that permutes the $k$ blocks according to $\sigma$. For example, if $\sigma = (1\ 2\ 3) \in \Sigma_3$, then $M_{\si \langle k_1, k_2, k_3 \rangle} $ is the block matrix
\[
\begin{pmatrix} 
0 & 0 & I_{k_3} \\
I_{k_1} & 0 & 0 \\
0 & I_{k_2} & 0
\end{pmatrix},\]
where $I_n$ denotes the $n\times n$ identity matrix. 

Let $\tau_j \in \Sigma_{k_j}$ for $j = 1, \dots, n$. Define the \emph{block sum} $\tau_1 \oplus ... \oplus \tau_n \in \Sigma_{k}$ to be the permutation that permutes via $\tau_j$ within the $j$-th block. Using permutation matrices, we have
\[M_{\tau_1 \oplus \dots \oplus \tau_n} = 
\begin{pmatrix} 
M_{\tau_1} & 0 & \cdots & 0 \\
0 & M_{\tau_2} & \cdots & 0 \\
\vdots & \vdots & \ddots & \vdots \\
0 & 0 & \cdots & M_{\tau_n} 
\end{pmatrix}.\]

We may combine these two constructions to define the \emph{permuted block sum} $\sigma \langle \tau_1,\ldots \tau_n \rangle$ as
\[
  \sigma\langle \tau_1,\ldots,\tau_n\rangle = \sigma\langle k_1,\ldots,k_n\rangle \cdot (\tau_1\oplus\cdots\oplus \tau_n)
\]
where $\sigma\in \Sigma_n$ and $\tau_i\in \Sigma_{k_i}$ for $i=1,\ldots,n$.  For instance,
\[
  M_{(1\ 2\ 3)\langle \tau_1,\tau_2,\tau_3\rangle} =
  \begin{pmatrix} 
0 & 0 & M_{\tau_3} \\
M_{\tau_1} & 0 & 0 \\
0 & M_{\tau_2} & 0
\end{pmatrix}.
\]

Finally, we take this opportunity to define two special classes of permutations which we will need to reference in our subsequent work.

\begin{defn}\label{defn:simple}
A permutation $\sigma \in \Sigma_n$ is \emph{simple} if, for any $k,k'\in \{1,\dots,n\}$ and $m$ such that $0<m<n-1$ and $m\leq n - \max\{k,k'\}$, $\sigma$ does not map $\{k, k+1, \dots, k+m\}$ to $\{k', k'+1,\dots, k'+m\}$. That is, $\sigma$ does not map any nontrivial proper interval to another nontrivial proper interval.  We call a permutation \emph{nonsimple} if it is not simple.
\end{defn}

\begin{ex}
The permutation $(2\ 3)\in \Sigma_3$ is nonsimple since it takes $\{2,3\}$ to itself.  The permutation $(1\ 2\ 4\ 3)\in \Sigma_4$ is simple.  Asymptotically, the fraction of simple permutations in $\Sigma_n$ is $\frac{1}{e^2}(1-\frac{4}{n})$ \cite{albert}.
\end{ex}

\begin{rem}
It is perhaps easiest to recognize a nonsimple permutation via its permutation matrix, which necessarily has a block decomposition with one block of size strictly between $1$ and $n$.
\end{rem}

We also need the notion of a separable permutation.

\begin{defn}\label{defn:separable}
The \emph{skew sum} of permutations $\sigma$ and $\tau$ is $(1\ 2)\langle \sigma,\tau\rangle$.  A permutation is \emph{separable} if it can be obtained from the trivial permutation $1_{\Sigma_1}$ by a finite number of block and skew sums.
\end{defn}

\begin{ex}
The permutation with matrix
\[
\begin{tikzpicture}
  \matrix(m)[matrix of math nodes,left delimiter=(,right delimiter=)]
    { &&&1&&&&\\
      &&1&&&&&\\
      &&&&1&&&\\
      1&&&&&&&\\
      &1&&&&&&\\
      &&&&&&1&\\
      &&&&&&&1\\
      &&&&&1&&\\
    };
\end{tikzpicture}
\]
is separable.
\end{ex}

\subsection{Operads}
The purpose of an operad is to encode families of operations.  We now provide a brief introduction to operads and their algebras here and set notation.

\begin{defn}\label{defn:operad}
  Let $\vV$ be a cartesian monoidal category. An
  \emph{operad} $\oO$ in $\vV$ consists of a sequence $\{\oO(n)\}_{n\geq 0}$ of
  objects in $\vV$ such that $\oO(n)$ has a right $\Si_n$-action,
  together with morphisms
  \[
  \ga \cn \oO(n) \times \oO(k_1) \times \cdots \times \oO(k_n) \rtarr 
  \oO(k_1 + \cdots +k_n)
  \]
  and 
  \[
  \beid \cn \ast \rtarr \oO(1),
  \]
  satisfying associativity, unitality and equivariance axioms. 
  See \cite{may72geo} or \cite{Yau2016Colored} for a complete description.  
  \end{defn}
  
  In this paper we will concentrate on operads in $\Set$, $\Cat$, $\GSet$ and $\GCat$, where $G$ is a finite group. Note that in these cases we can think of $\beid$ as a ($G$-fixed) element, respectively object, in the ($G$-)set, respectively ($G$-)category, $\oO(1)$. If $f\in \oO(n)$ we say $f$ has \emph{arity} $n$ and write $|f|=n$. In the case of $\GSet$ and $\GCat$, we often think of $\oO(n)$ as a left $G\times \Si_n$-object via $(g,\si)\cdot f = g\cdot f \cdot \si^{-1}$.
  
  Elements of $\oO(n)$ should be thought of as operations with $n$ inputs and $1$ output, so as such, they will be depicted as trees, with $\ga$ depicted as grafting. For example, if $f\in \oO(2)$, $g_1\in \oO(3)$, and $g_2\in \oO(1)$, we depict $\ga(f;g_1,g_2)\in \oO(4)$ as
  \begin{center}
\begin{tikzpicture}[scale=0.75]
\begin{scope}[xshift=9pt,yshift=-5in,grow'=up,
frontier/.style={distance from root=150pt}]

	\Tree [.$f.$ [.$g_1$ [.$\,$ ] [.$\,$ ] [.$\,$ ] ] [.$g_2$ [.$\,$ ]] ]

\end{scope}
\end{tikzpicture}

\end{center}

Associativity of $\ga$ can then be interpreted as saying that the grafting of three levels can be done in any order, yielding the same result. Thus the tree
  \begin{center}
\begin{tikzpicture}[scale=0.75]
\begin{scope}[xshift=9pt,yshift=-5in,grow'=up,
frontier/.style={distance from root=150pt}]

	\Tree [.$f$ [.$g_1$ [.$h_{11}$ [.$\,$ ] [.$\,$ ] ] [.$h_{12}$ [.$\,$ ] ] [.$h_{13}$ [.$\,$ ] ] ] [.$g_2$ [.$h_{21}$ [.$\,$ ] [.$\,$ ] [.$\,$ ] ] ] ]

\end{scope}
\end{tikzpicture}

\end{center}
has a unique interpretation as 
\[\ga(f;\ga(g_1;h_{11},h_{12},h_{13}),\ga(g_2;h_{21}))=\ga(\ga(f;g_1,g_2);h_{11},h_{12},h_{13},h_{21}).\]

For $m\geq 1$, $n\geq 0$ and $1\leq i \leq m$, we define the $i$th \emph{partial composition}
\[\circ _i\cn \oO(m) \times \oO(n) \rtarr \oO(m+n-1)\]
as the composite
\[ \oO(m) \times \oO(n) \to \oO(m) \times \oO(1)^{i-1} \times \oO(n) \times \oO(1)^{m-i} \xrightarrow{\ga} \oO(m+n-1), \]
where the first arrow is induced by the map $\beid\cn \ast \to \oO(1)$.
In terms of elements, the $i$th partial composition is given by
\[f\circ_i g =\ga (f; \beid, \dots, \beid, g, \beid, \dots, \beid),\]
where $g$ is in the $i$th position of the $m$-tuple. It should be thought of as grafting $g$ onto the $i$th leaf of $f$ and prolonging the rest of the leaves appropriately.
    
  \begin{defn}\label{defn:O-algebra}
  An \emph{$\oO$-algebra} in $\vV$ is given by a pair $(X,\mu)$, where $X$ is an
  object of $\vV$, and $\mu$ is a collection of morphisms
  \[
  \mu_n \cn \oO(n) \times X^{n} \rtarr X
  \]
  in $\vV$ satisfying equivariance conditions and compatibility with
  $\ga$ and $\beid$. See \cite[\S13.2]{Yau2016Colored} for a complete list of axioms.
  \end{defn}

\begin{rem}\label{rem:2cats}
For a given operad $\oO$, there is a notion of maps between $\oO$-algebras which we do not describe here as it is not the focus of this paper \cite[Definition 13.2.8]{Yau2016Colored}. Moreover, if $\oO$ is an operad in $G\Cat$, there are two other notions of maps: lax and pseudo, which satisfy compatibilities up to natural transformations and natural isomorphisms, respectively (see \cite[Definition 2.21]{RubinThesis}). In this case, the 2-categorical nature of $G\Cat$ also implies that there is a notion of transformation between algebra maps \cite[Definition 2.27]{RubinThesis}. We thus have three relevant 2-categories with objects given by $\oO$-algebras, with 1-morphisms given respectively by strict, pseudo, and lax maps. In all three cases the 2-morphisms are given by these transformations.
\end{rem}

 \begin{defn}\label{defn:map_operad}
 Let $\nN$ and $\oO$ be operads in $\vV$. A \emph{map of operads} $f\colon \nN \rightarrow \oO$ consists of a $\Sigma_n$-equivariant morphism $f_n \colon  \nN(n) \rightarrow \oO(n)$ for all $n\geq 0$, such that they respect the unit and the operadic composition.
 
 When $\vV=G\Cat$, we say $f$ is an \emph{equivalence} if the map of fixed points \[f_n^\Gamma \colon \nN(n)^\Gamma \rightarrow \oO(n)^\Gamma\] is a weak equivalence on passage to classifying spaces for all $\Gamma \le G \times \Sigma_n$.
 \end{defn}

We note that if $F\colon \vV \to \wW$ is a product-preserving functor and $\oO$ is an operad in $\vV$, then $F(\oO)$ will form an operad in $\wW$ with all the structure induced from that of $\oO$ (see \cite[Theorem 11.5.1]{Yau2016Colored}  for a more general version of this result). Most of the operads used in this paper will be constructed this way from the following example.

\begin{ex}
The \emph{associativity operad} in $\Set$ is given by the sequence $\assoc(n)=\Si_n$, with (right) $\Si_n$-action given by right multiplication. The composition 
\[ \ga\cn \Si_n \times \Si_{k_1} \times \dots \times \Si_{k_n} \rtarr \Si_{k_1+ \dots + k_n}\]
is given by
\[\ga(\si; \ta_1,\dots,\ta_k)=\si\langle \ta_1,\dots,\ta_n\rangle.\]
The identity $\beid$ is given by $1_{\Si_1}\in \Si_1$.
Algebras over $P$ in $\Set$  are (unbiased) associative and unital monoids.
\end{ex}

\subsection{The categorical Barratt-Eccles operad and its equivariant analogue}
The categorical Barratt-Eccles operad plays an important role in the theory of infinite loop spaces. To construct it, we first recall the chaotic category functor $\chao{(-)} \cn \Set \to \Cat$. 

\begin{defn}
 Given a set $X$, we denote by $\chao{X}$ the category with objects given by $X$ and a unique morphism between any two objects. It is called the \emph{chaotic category on $X$}. 
\end{defn}

The construction above extends to a functor $\chao{(-)}\cn \Set \to \Cat$ that is right adjoint to object functor $\Cat \to \Set$. As a right adjoint, $\chao{(-)}$ preserves products and hence, sends operads in $\Set$ to operads in $\Cat$.

\begin{defn}
The \emph{categorical Barratt-Eccles operad} $\cbe$ is the operad in $\Cat$ defined as $\chao{\assoc}$. In particular, $\cbe(n)=\chao{\Si_n}$.
\end{defn}

We recall the definition of a permutative category.

\begin{defn}\label{dfn:permcat}
 A \emph{permutative category} consists of 
 \begin{itemize}
  \item a category $\cC$;
  \item an object $e\in \cC$;
  \item a functor $\otimes \colon \cC \times \cC \to \cC$;
  \item a natural isomorphism 
  \[
    \begin{tikzpicture}[x=1mm,y=1mm]
    \draw[tikzob,mm] 
    (0,0) node (00) {\cC \times \cC}
    (25,0) node (10) {\cC \times \cC}
    (12.5,-10) node (01) {\cC}
    ;
    \path[tikzar,mm] 
    (00) edge node {\tau} (10)
    (10) edge node {\otimes} (01)
    (00) edge[swap] node {\otimes} (01)
    ;
    \draw[tikzob,mm]
    (12.5,-4) node {\Anglearrow{40} \beta}
    ;
  \end{tikzpicture}
  \]
   called the \emph{symmetry}, whose components are given by morphisms $\beta_{a,b}\colon a \otimes b \to b \otimes a$ in $\cC$.
    \end{itemize}
 The data above are subject to the following axioms
 \begin{enumerate}[(i)]
  \item $e$ is a strict two-sided unit for $\otimes$ that is, for all $a\in \cC$,
  \[e\otimes a = a = a \otimes e ;\]
  \item $\otimes$ is strictly associative: for all $a,b,c\in \cC$,
  \[a\otimes(b\otimes c) = (a \otimes b)\otimes c ;\]
  \item for all $a,b,c\in \cC$, the following diagrams commute
  \[  \begin{tikzpicture}[x=30mm,y=15mm]
    \draw[tikzob,mm] 
    (0,-1) node (0) {a\otimes e}
    (1,-1) node (1) {e\otimes a}
    (.5,0) node (2) {a};
    \path[tikzar,mm] 
    (0) edge[swap] node {\beta_{a,e}} (1)
    (0) edge[/tikz/commutative diagrams/equal] (2)
    (2) edge[/tikz/commutative diagrams/equal] (1);
  \end{tikzpicture}
  \qquad
  \begin{tikzpicture}[x=30mm,y=15mm]
    \draw[tikzob,mm] 
    (0,-1) node (0) {a\otimes b}
    (1,-1) node (1) {a\otimes b}
    (.5,0) node (2) {b\otimes a};
    \path[tikzar,mm] 
    (0) edge[swap] node {\id} (1)
    (0) edge node {\beta_{a,b}} (2)
    (2) edge node {\beta_{b,a}} (1);
  \end{tikzpicture}\]
\[
    \begin{tikzpicture}[x=30mm,y=15mm]
    \draw[tikzob,mm] 
    (0,-1) node (0) {a\otimes b \otimes c}
    (1,-1) node (1) {c\otimes a\otimes b}
    (.5,0) node (2) {a\otimes c\otimes b};
    \path[tikzar,mm] 
    (0) edge[swap] node {\beta_{a\otimes b,c}} (1)
    (0) edge node {\id \otimes \beta_{b,c}} (2)
    (2) edge node {\beta_{a,c}\otimes \id} (1);
  \end{tikzpicture}
  \]
 \end{enumerate}
\end{defn}

As noted in the introduction, algebras over $\cbe$ are in one-to-one correspondence with permutative categories \cite{May1974Einfty} with $e$ and $\otimes$ represented by $1_{\Si_0}$ and $1_{\Si_2}$, respectively, and $\beta$ represented by the unique morphism in $\cbe(2)$ from $1_{\Si_2}$ to $(1\ 2)$.

For a finite group $G$, we also define the categorical $G$-equivariant Barratt-Eccles operad.  We use the functor $\Set(G,-)\cn \Set \to \GSet$ that takes a set $X$ to the set $\Set(G,X)$ of all functions from $G$ to $X$ with left $G$-action given as follows. For $g\in G$ and $f\in \Set(G,X)$, the function $g\cdot f$ sends $h\in G$ to $f(hg)$. This is a product-preserving functor, and as such we can use it to transfer operads from $\Set$ to $\GSet$.

\begin{defn}\label{P_G:defn}
 The operad $\obe{G}$ in $\GSet$ is defined as $\Set(G,\assoc)$. In particular, an element in $\obe{G}(n)$ is a function (not necessarily a group homomorphism) $f\cn G \to \Si_n$. The \emph{categorical $G$-equivariant Barratt-Eccles operad} $\be{G}$ is the operad in $\GCat$ defined as $\chao{\obe{G}}$. Algebras over $\be{G}$ are called \emph{permutative $G$-categories.} 
\end{defn}

\begin{rem}
 A standard calculation shows that the operad $\be{G}$ can be alternatively defined as the hom category $\Cat(\chao{G},\chao{\assoc})$. As noted in \cite[Example 3.8]{RubinRealization}, the operad $\be{G}$ is isomorphic but not equal to the one defined in \cite{GM}, the main difference being that the $G$-actions are slightly different. There the authors prove that upon geometric realization, one obtains an $E_\infty$ $G$-operad in $\GTop$. 
\end{rem}

\subsection{Presentations for operads in $G\Set$}

We conclude this section by recalling how presentations of operads work. The basic characters are free operads (in $G$-sets) and quotients.  In \cite[Construction 7.6]{RubinRealization}, Rubin presents a model for the free symmetric $G$-operad on a sequence of $G$-sets.  Our starting point is a sequence of sets, from which we build the sequence of free $G$-sets and then apply Rubin's construction.  This significantly simplifies the construction, as indicated in \cite[Proposition 8.2]{RubinRealization}.

\begin{defn}\label{defn:free}
Let $S = \{S_n\}$ be a sequence of sets and let $G\times S= \{G\times S_n\}$ denote the induced sequence of free $G$-sets.  The \emph{free operad} in $G$-sets on $S$, denoted $\free{S}$, is $F_0(G\times S)$ in the notation of \cite[\S7]{RubinRealization}.
\end{defn}

We may interpret the elements of $\free{S}(n)$ as isomorphism classes of finite rooted planar trees with $n$ leaves and $k$-ary nodes labeled by elements of $G\times S_k$; the entire $(G\times S)$-labeled tree is then further labeled by an element of $\Sigma_n$.  The $\Sigma$-action is the obvious one, operadic composition is given by grafting trees, and the $G$-action simply multiplies the $G$-label of each node.

We now move on to quotients of free operads, following \cite[\S6]{RubinRealization}.

\begin{defn}\label{defn:reln}
Let $\oO$ be an operad in $G$-sets.  A \emph{congruence relation} on $\oO$ is a graded equivalence relation $\sim ~= (\sim_n)_{n\ge 0}$ which respects the $G\times \Sigma$-action and operadic composition.  If $R = (R_n)_{n\ge 0}$ is a graded binary relation on $\oO$, then the smallest congruence relation containing $\oO$, denoted $\langle R\rangle$, is the \emph{congruence relation generated by $R$}.
\end{defn}

\begin{rem}
Congruence relations are closed under intersection, and $R=(\oO(n))_{n\geq 0}$ constitutes a congruence relation, so we may construct $\langle R\rangle$ by taking the intersection of all congruence relations containing $R$.
\end{rem}

We can form quotients of operads in $G\Set$ by congruence relations satisfying the expected universal property; see \cite[Proposition 6.5]{RubinRealization}.  To be specific, given an operad $\oO$ and a congruence relation $\sim$, there is a $G$-operad $\oO/\sim$ and operad map $\oO\to \oO/\sim$ such that any other operad map $\oO\to \oO'$ which respects $\sim$ factors uniquely through $\oO/\sim$.

Finally, we note the following proposition which will be important when we pass from operads in $G$-sets to operads in $G$-categories via the chaotic functor.

\begin{prop}[{\cite[Proposition 4.12]{RubinThesis}}]\label{prop:chaoticQuotient}
Suppose that $\oO$ is an operad in $G$-sets, $R$ is a binary relation on $\oO$, and $f:\chao{\oO}\to \nN$ is a map of operads in $G$-categories.  Then $f$ factors through the quotient map $\chao{\oO}\to \chao{\oO}/\langle R\rangle = \chao{\oO/\langle R\rangle}$ if and only if $f$ respects $R$ on objects and $f(x\to y) = \Id:f(x)\to f(y)$ whenever $xRy$.  In such a case, the induced map $\chao{\oO}/\langle R\rangle\to \nN$ is unique.
\end{prop}

See \cite[Proposition 4.13]{RubinThesis} for an enhancement of this result to the $2$-category of algebras and lax maps.

\section{The $G$-Barratt-Eccles operad is not finitely generated}\label{sec:gen}
Recall the operad $\obe{G}=\Set(G,\assoc)$ from \cref{P_G:defn}. In this section we prove that $\obe{G}$ is not finitely generated for $|G|>1$. To do so, we introduce the notion of the suboperad generated by a sequence of subsets.

Recall that $\mathcal{N}\subseteq \mathcal{O}$ is a suboperad if $\mathcal{N}(n)\subseteq \mathcal{O}(n)$ is a $\Si_n$-subset for all $n$, $\beid\in \mathcal{N}(1)$, and $\mathcal{N}$ is closed under the operadic composition for $\mathcal{O}$.

\begin{defn}\label{defn:gen}
For an operad $\mathcal{O}$ in $\GSet$ and a sequence of subsets $S=\{S_i : S_i\subseteq \mathcal{O}(i)\}$, the suboperad generated by $S$, denoted $\langle S\rangle$ is the smallest suboperad of $\mathcal{O}$ such that $S_i\subseteq \langle S\rangle(i)$ for all $i$.  The operad $\oO$ is called \emph{finitely generated} if there exists such an $S$ with $\left|\coprod_{i\in \NN}S_i\right|<\infty$ and $\langle S\rangle = \oO$.
\end{defn}

\begin{rem}\label{rem:below}
The definition permits $S_i$ to be empty, and it is necessary that $S_i=\varnothing$ for sufficiently large $i$ in order for $S$ to witness finite generation of $\oO$.
\end{rem}

The reader may check that we may explicitly model $\langle S\rangle$ in the following fashion.

\begin{prop}\label{prop:genModel}
If $\mathcal{O}$ is an operad in $G$-sets and $S=\{S_i : S_i\subseteq \mathcal{O}(i)\}$, then $\langle S\rangle(k)$ is the set of $\Sigma_k$-actions on operadic compositions of $G$-actions on elements of $S$, \emph{i.e.},
\[
  \langle S\rangle(k) = \left\{((g_0s_0)\circ_{i_1} (g_1s_1) \circ_{i_2} \cdots \circ_{i_m} (g_ms_m))\cdot \sigma
  ~\middle|~
  \begin{array}{c}
  m\in \NN, g_i \in G, s_i\in S_{k_i},\\
  \sum k_i = k+m, \sigma\in \Sigma_k
  \end{array}
  \right\}.
\]
\end{prop}

\begin{rem}\label{rem:surj}
 Note that taking the elements of $S$ as abstract symbols, one can construct a surjective map of operads $\free S \to \langle S \rangle $. Thus by \cite[Corollary 6.7]{RubinRealization}, one can construct the latter as a quotient of the former by the kernel of this map.
\end{rem}

We now introduce two further notions of generation that will be important in our proof that $P_G$ is not finitely generated.

\begin{defn}
An element $f \in \obe{G}(n)$ is \emph{$\gamma$-generated from below}  if there exist  
$s, h_1, \dots, h_{|s|}$ $\in \bigcup_{i=0}^{n-1} \obe{G}(i)$ such that
\[
  f=\gamma(s; h_1, \dots, h_{|s|}).
\]
An element $f' \in \obe{G}(n)$ is \emph{generated from below} if it is of the form $f\cdot \sigma$ for $f$ $\gamma$-generated from below and $\sigma\in \Sigma_n$.
\end{defn} 

\begin{rem}\label{rem:prim-unique}
The $G$-equivariance axiom on $\ga$ guarantees that the set of elements of arity $n$ that are generated from below is closed under the $G$-action.
\end{rem}

We now consider how the notions of $\gamma$-generation from below and generation from below interact with a special class of elements of $\obe{G}$, the primitive ones:

\begin{defn}
Call $f\in \obe{G}(n)$ \emph{primitive} if $f(1_G)=1_{\Sigma_n}$. If $f$ is not primitive, we call it \emph{nonprimitive}.
\end{defn}

\begin{rem}
For each $f\in \obe{G}(n)$, the permutation $f(1_G)^{-1}$ is the unique $\sigma\in \Sigma_n$ such that $f\cdot \sigma$ is primitive.
\end{rem}

\begin{lemma}\label{lem:gen-prim}
Suppose $f$ is $\gamma$-generated from below with $f=\gamma(s; h_1, \dots, h_{|s|})$, $|s|;|h_1|,\ldots, |h_{|s|}|<n$. Then $f$ is primitive if and only if it can be written as $f=\gamma(s'; h_1', \dots, h_{|s'|}')$, where $s',h_1', \dots, h_{s'}$ are all primitive elements of arity greater than 0 and less than $n$.
 \end{lemma}

\begin{proof}
The reverse direction is trivial.

For the forward direction, suppose  $f=\gamma(s; h_1, h_2,\dots, h_{|s|})$. Let $h_1',h_2',\dots$ be the terms of $h_1,\dots,h_{|s|}$ not equal to $e$ in ascending order, and let
\[s'=\gamma(s;t_1,\dots,t_{|s|}) \quad \text{where} \quad t_j=
\begin{cases}
\beid & \text{if } h_j\neq e\\
e & \text{if } h_j=e.
\end{cases}\]
Thus by associativity and unitality of $\gamma$, we have that $f=\gamma(s'; h_1', h_2',.. h_{|s'|}')$ and hence, 
\[1_{\Sigma_{|f|}}=f(1_G)=s'(1_G) \langle h_1'(1_G), \dots,  h_{|s'|}'(1_G)\rangle.\]
Note that if $h_i(1_G)$ is not the identity for some $i$, then the expression on the right hand side cannot be the identity. Since all of the $h_i$ are of arity at least 1, $s'(1_G)$ must also be the identity, as desired.
\end{proof}

\begin{lemma}\label{lemma:prim}
If $f\in \obe{G}(n)$ is primitive and generated from below (but not necessarily $\gamma$-generated from below), then $f$ can be written in the form $f=\gamma(s; h_1, \dots, h_{|s|})$ where $s, h_1, \dots, h_{|s|}$ are all primitive and of arity less than $n$. 
\end{lemma}

\begin{proof}
Suppose some primitive $f\in \obe{G} (n)$ is generated from below, meaning
\[
  f=\gamma(s; h_1, h_2,.. h_{|s|})\cdot \sigma
\]
for some $\sigma \in \Sigma_{n}$ and $|s|;|h_1|,\ldots,|h_{|s|}|<n$. Then let $\rho=s(1_{G})$, $s'=s\cdot \rho^{-1}$, $\ta_i=h_i(1_G)$, and $h_{i}'= h_{i}\cdot \ta_i^{-1}$ for $i=1,\dots, |s|$. Note that $s_i';h_1', \dots, h_{|s|}'$ are all necessarily primitive. It follows from the $\Sigma$-equivariance axioms of an operad that
\[
  f=\gamma(s';h_{\rho^{-1}(1)}', \dots, h_{\rho^{-1}(|s|)}')\cdot(\rho \langle \ta_1,\dots, \ta_{|s|} \rangle \sigma).
\]

Let $\sigma'= \rho \langle \ta_1,\dots, \ta_{|s|} \rangle \sigma$ and $f'= \gamma(s';h_{\rho^{-1}(1)}', \dots,h_{\rho^{-1}(|s|)}' )$. Then $f'$ is primitive, since all the arguments are primitive. Then $f=f'\cdot \sigma'$. Since $f$ is also primitive, we know that $\sigma'=1_{\Sigma_{|f|}}$ and thus $f=f'$, which is the desired form. 
\end{proof}

Recall the notion of a \emph{nonsimple permutation} from \cref{defn:simple}.

\begin{lemma}\label{lemma:nonsimple}
Suppose $f\in \obe{G}(n)$ is primitive and $\gamma$-generated from below. Then $f(g)$ is nonsimple for all $g\in G$.
\end{lemma}

\begin{proof}
Assume $f=\gamma(s;h_1,\dots,h_{|s|})$ with the arity of all arguments less than $n$. By \cref{lem:gen-prim}, we may suppose $|h_i|\geq 1$ for all $i$. Moreover, since we assumed $f$ is generated from below, $1<|h_k|<n$ for some $k$. Fix an arbitrary $g\in G$. Then 
\[f(g)=s(g)  \langle h_1(g), \dots, h_{|s|}(g) \rangle.\]
Thus, by definition, $f(g)$ permutes the intervals of length $|h_1|,\dots,|h_{|s|}|$ according to $s(g)$, making it nonsimple.
\end{proof}

\begin{prop}
If an element $f\in \obe{G} (n)$ is generated from below, then $f(g) f(1_{G})^{-1}$ is a nonsimple permutation for all $g\in G.$
\end{prop}
\begin{proof}
Suppose $f\in \obe{G} (n)$ is generated from below, and let $f'=f\cdot f(1_{G})^{-1}$. Then $f'$ is primitive and generated from below, so by \cref{lemma:prim} it is $\ga$-generated from below. Thus, by \cref{lemma:nonsimple}, $f'(g)=f(g) f(1_{G})^{-1}$ is nonsimple for all $g\in G$.
\end{proof}

We can now prove that $\obe{G}$ is not finitely generated for nontrivial groups $G$.

\begin{thm}\label{cor:PG-not-fin-gen}
Let $G$ be a nontrivial finite group. For $n>3$, there exists at least one element $f\in \obe{G}(n)$ such that $f$ is not generated from below. In particular, the operad $\obe{G}$ is not finitely generated for $|G|>1$.
\end{thm}
\begin{proof}
Note that at any arity $n>3$, there exists at least one simple permutation. Therefore, at any arity $n>3$, there exists a primitive element $f\in \obe{G}$ and $g\in G\smallsetminus\{1_G\}$ such that $f(g)$ is simple, and so $f$ is not generated from below.

In any candidate finite generating set $S$, there is an element of highest arity. Call this highest arity $m$. Let $m'=m+4$ (to ensure $m'>3$). Then $\langle S\rangle \subseteq\langle\bigcup_{i=0}^{m'-1} P_{G} (i) \rangle$. There exists an element in $P_{G} (m')$ that is not generated from below, and therefore not generated by $S$, so $\obe{G}$ is not finitely generated.
\end{proof}

\section{Finitely generated $E_\infty$ $G$-operads}\label{sec:biased}
We have seen that the object operad $\obe{G}$ of the $G$-equivariant Barratt-Eccles operad is not finitely generated for nontrivial $G$.  This makes the task of explicitly describing $\be{G}$-algebras with a finite amount of data seem intractable.  Fortunately, $\obe{G}$ admits finitely generated suboperads $\osubbe{G}\subseteq \obe{G}$ such that $\subbe{G} := \chao{\osubbe{G}}\simeq \be{G}$.  In this section, we introduce the operads $\osubbe{G}$ for arbitrary $G$, and then give explicit presentations of $\osubbe{G}$ for $G$ a cyclic group of order $2$ or $3$.

\subsection{The operads $\subbe{G}$}\label{subsec:QG}
Fix a finite group $G$.  Morally speaking, the suboperad $\osubbe{G}$ is generated by the operations $e\in P_G(0)$, $\otimes\in P_G(2)$ (the constant function at $1_{\Sigma_2}$), and norms for all $G$-orbits.  In order to make the last notion precise, we make the following three definitions.

\begin{defn}[{cf. \cite[Definition 2.5]{RubinRealization}}]\label{defn:norm}
Given a finite ordered $G$-set $T$, write $\otimes_T:G\to \Sigma_{|T|}$ for the permutation representation of $T$.  Considered as an element of $P_G(|T|)$, we call $\otimes_T$ an \emph{external norm for $T$}.
\end{defn} 

Note that $e$ is the norm for $\varnothing$ and $\otimes$ is the norm for any $2$-element set with trivial $G$-action.

\begin{defn}\label{defn:complete}
For a finite group $G$, let $\frN$ be a set of ordered transitive $G$-sets.  Call $\frN$ a \emph{complete set of ordered $G$-orbits} if it contains exactly one non-trivial transitive $G$-set of each isomorphism class (forgetting ordering).
\end{defn}

Clearly, we may produce a complete set of ordered $G$-orbits by arbitrarily ordering each $G/H$ as $H$ ranges through a set of representatives of conjugacy classes of proper subgroups of $G$.

\begin{defn}\label{defn:QG}
Given a complete set of ordered $G$-orbits $\frN$, let
\[
  \osubbe{G,\frN} := \langle \otimes_T\mid T\in \frN\cup \{\varnothing,\{0,1\}\}\rangle \subseteq \obe{G}
\]
where $\{0,1\}$ has trivial $G$-action and $0<1$. We define $\subbe{G,\frN}$ as the chaotic operad on $\osubbe{G,\frN}$, and we note that it is a full suboperad of $\be{G}$. We call $\subbe{G,\frN}$ the \emph{complete suboperad of $\be{G}$ relative to $\frN$}. 
If the choice of $\frN$ is understood from context, then we will write $\subbe{G}$ for $\subbe{G,\frN}$ and call it a \emph{complete suboperad of $\be{G}$}.
\end{defn}

Since any complete set of ordered $G$-orbits is finite, the operads $\osubbe{G,\frN}$ are finitely generated.  They also have the following remarkable property.

\begin{thm}\label{thm:complete}
For any finite group $G$ and complete set of ordered $G$-orbits $\frN$, $\subbe{G,\frN}$ is an $E_\infty$ $G$-operad and the inclusion $\subbe{G,\frN}\hookrightarrow \be{G}$ is an equivalence of $G$-operads.
\end{thm}

\begin{proof}
The operad $\osubbe{G,\frN}$ is $\Si$-free since it is a suboperad of $\obe{G}$, and it is a quotient of $\free(\{\frN\cup \{\varnothing,\{0,1\}\})/\langle R\rangle$ where $R$ encodes the relations $g\cdot \varnothing = \varnothing$, $g\cdot \{0,1\}=\{0,1\}$, and $g\cdot G/H = G/H\cdot \sigma_{G/H}(g)$ for $g\in G$ and $H<G$. As noted in \cite[Example 6.5]{RubinThesis}, $\free(\{\frN\cup \{\varnothing,\{0,1\}\})/\langle R\rangle$ is an $E_\infty$ $G$-operad after applying $\widetilde{(~)}$, i.e., it is an $N_\infty$ operad that contains all norms. Thus, the same is true for $\subbe{G,\frN}$ and the result follows.
\end{proof}

 \begin{rem}
 The equivalence of operads $\subbe{G,\frN}\to\be{G}$ induces a biequivalence between their respective 2-categories of algebras and lax maps (see \cite[Theorem 4.21]{RubinThesis}).
 \end{rem}

\subsection{Presentation for $Q_{C_2}$}
In this subsection, we specialize to $G=C_2$, which we consider to have generator $g$.  Note that the operad $\osubbe{C_2}$ has three generators:
\begin{enumerate}
  \item $e\in \obe{C_2}(0)$,
  \item $\otimes\in \obe{C_2}(2)$, which is the function constant at the identity permutation, and
  \item $\boxtimes \in \obe{C_2}(2)$, which sends $1_{C_2}$ to the identity permutation $1_{\Si_2}$ and $g$ to the permutation $(1\ 2)$.
\end{enumerate}

We thus have a map of operads $\phi: \free\{e,\otimes,\boxtimes\}\to \osubbe{C_2}$ determined by sending each generator to its namesake.  The following theorem gives a presentation for $\osubbe{C_2}$.

\begin{thm}\label{thm:pres-c2}
The operad $\osubbe{C_2}$ is isomorphic to $\free{\{e,\otimes,\boxtimes\}}/\langle R \rangle$, where $R$ consists of the following:
\begin{enumerate}
    \item Strict unit: 
    \[\gamma(\boxtimes; \beid, e) = \gamma(\boxtimes;e,\beid) = \gamma(\otimes; \beid,e) = \gamma(\otimes; e, \beid) = \beid;\]
    
    \item Strict associativity: for any primitive $\dia \in \osubbe{C_2}(2),$
    \[ \gamma(\dia; \dia, \beid) = \gamma(\dia; \beid, \dia);\]
    
    \item Group action: $g \cdot \boxtimes = \boxtimes \cdot (1\ 2)$, $g\cdot \otimes = \otimes$, and $g\cdot e = e$.
  \end{enumerate}
\end{thm}

In order to prove the theorem we will use tree representations for elements in the operad $\oO=\free{\{e,\otimes,\boxtimes\}}/\langle R \rangle$. Let $q\colon \free{\{e,\otimes,\boxtimes\}} \to \oO$ denote the quotient map.

\begin{defn}\label{treerep}
Let $\tree{n}\subset \free{\{e,\otimes,\boxtimes\}}(n)$ be the set of planar rooted trees with nodes labeled by $(1_G,e)$, $(1_G,\otimes)$ and $(1_G,\boxtimes)$, i.e., without using the $G$ and $\Sigma$ actions. For simplicity, we will refer to these nodes by their second coordinate. 

Let $\cantree{n}\subseteq \tree{n}$ denote the elements of $\free{\{e,\otimes,\boxtimes\}}(n)$ derived from trees with only $\otimes$ and $\boxtimes$ nodes and such that no instance of $\otimes$ is grafted directly to the right branch of another instance of $\otimes$, and similarly, no instance of $\boxtimes$ is grafted directly to the right branch of another instance of $\boxtimes$; we call elements of $\cantree{n}$ \emph{canonical trees}.
\end{defn}

\begin{rem}\label{treerotation}
Given an arbitrary tree $T\in \tree{n}$, we can get a canonical tree $T_c$ with $q(T_c)=q(T)$ by first replacing every instance of
\[
\begin{tikzpicture}[scale=0.75]
\begin{scope}[xshift=9pt,yshift=-5in,grow'=up,
frontier/.style={distance from root=150pt}]
\Tree [.$\otimes$ 
[.$T'$ ]
[.$e$ ] ]
\end{scope}
\end{tikzpicture} 
\qquad
\begin{tikzpicture}[scale=0.75]
\begin{scope}[xshift=9pt,yshift=-5in,grow'=up,
frontier/.style={distance from root=150pt}]
\Tree [.$\otimes$ 
[.$e$ ]
[.$T'$ ] ]
\end{scope}
\end{tikzpicture}
\qquad 
\begin{tikzpicture}[scale=0.75]
\begin{scope}[xshift=9pt,yshift=-5in,grow'=up,
frontier/.style={distance from root=150pt}]
\Tree [.$\boxtimes$ 
[.$T'$ ]
[.$e$ ] ]
\end{scope}
\end{tikzpicture}
\qquad
\begin{tikzpicture}[scale=0.75]
\begin{scope}[xshift=9pt,yshift=-5in,grow'=up,
frontier/.style={distance from root=150pt}]
\Tree [.$\boxtimes$ 
[.$e$ ]
[.$T'$ ] ]
\end{scope}
\end{tikzpicture}
\]
by $T'$ itself, using the first relation. We then obtain $T_c$ by rotating nodes to the left when possible, i.e., replace every instance of 
\[
\begin{tikzpicture}[scale=0.75]
\begin{scope}[xshift=9pt,yshift=-5in,grow'=up,
frontier/.style={distance from root=150pt}]
\Tree [.$\dia$ 
[.$T_1$ ]
[.$\dia$ 
[.$T_2$ ] [.$T_3$ ] ] ]
\end{scope}
\end{tikzpicture}  \parbox[c][.8in][t]{2cm}{\begin{center}\text{with}\end{center}}  
\begin{tikzpicture}[scale=0.75]
\begin{scope}[xshift=9pt,yshift=-5in,grow'=up,
frontier/.style={distance from root=150pt}]
\Tree [.$\dia$ 
[.$\dia$ 
[.$T_1$ ] 
[.$T_2$ ] ] [.$T_3$ ] ]
\end{scope}
\end{tikzpicture}
\] 

\vspace{-.9cm} \noindent where $\dia$ is either $\otimes$ or $\boxtimes$ and $T_1, T_2, T_3$ are trees. For example, the process can be described visually as follows:

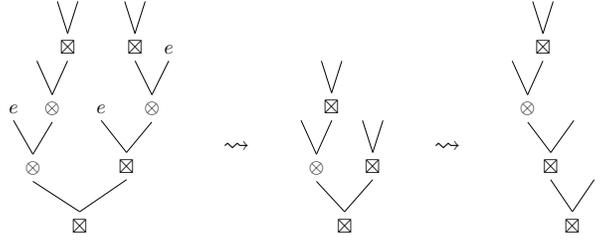
\[
\begin{tikzpicture}[scale=0.75]
\begin{scope}[xshift=9pt,yshift=-5in,grow'=up,
frontier/.style={distance from root=150pt}]
\Tree [.$\boxtimes$ 
[.$\otimes$ [.$e$ ] 
    [.$\otimes$ [.$\,$ ] 
        [.$\boxtimes$ [.$\,$ ] [.$\,$ ] ] ] ]
[.$\boxtimes$ [.$e$ ] 
    [.$\otimes$  
        [.$\boxtimes$ [.$\,$ ] [.$\,$ ] ]
        [.$e$ ] ] ] ]
\end{scope}
\end{tikzpicture}
\parbox[c][1in][t]{1.5cm}{\begin{center}$\leadsto$\end{center}}
\begin{tikzpicture}[scale=0.75]
\begin{scope}[xshift=9pt,yshift=-5in,grow'=up,
frontier/.style={distance from root=150pt}]
\Tree [.$\boxtimes$ 
[.$\otimes$ [.$\,$ ] 
        [.$\boxtimes$ [.$\,$ ] [.$\,$ ] ] ]
[.$\boxtimes$ [.$\,$ ] [.$\,$ ] ] ] 
\end{scope}
\end{tikzpicture}
\parbox[c][1in][t]{1.5cm}{\begin{center}$\leadsto$\end{center}}
\begin{tikzpicture}[scale=0.75]
\begin{scope}[xshift=9pt,yshift=-5in,grow'=up,
frontier/.style={distance from root=150pt}]
\Tree [.$\boxtimes$ 
[.$\boxtimes$ 
  [.$\otimes$ 
    [.$\,$ ] [.$\boxtimes$ [.$\,$ ] [.$\,$ ] ] ]
  [.$\,$ ] ] [.$\,$ ] ] 
\end{scope}
\end{tikzpicture}.
\]

\vspace{-.8cm} This process always ends in a canonical tree. At this point we do not claim that $T_c$ is unique, although that will be a consequence of the proof of \cref{thm:pres-c3} below.
\end{rem}

\begin{proof}[Proof of \cref{thm:pres-c2}]
The reader can check that all the relations in $R$ are indeed satisfied in $\obe{C_2}$, and hence in $\osubbe{C_2}$. Thus, by \cref{rem:surj}, the map $\phi$ induces a level-wise surjective map $\phi:\oO\to \osubbe{C_2}$.  Call an element of $\oO(n)$ \emph{primitive} when it has a representative in $\tree{n}$; recall that an element $f\in\osubbe{C_2}(n)$ is primitive if $f(1_G)=1_{\Sigma_n}$. Note that $\phi$ restricts to primitive elements. We will first show that it suffices to prove that $\phi$ induces a bijection $\prim \oO(n)\to \prim \osubbe{C_2}(n)$ for all $n$.

Indeed, by the equivariance axiom and relation (3), every element in $\oO(n)$ is of the form $f\cdot \si$ for some primitive $f$ and some $\si\in\Si_n$. If $\phi(f\cdot \si)=\phi(f'\cdot \si')$ for $f$ and $f'$ primitive, then $\phi(f)\cdot \sigma=\phi(f')\cdot \sigma'$ by $\Sigma$-equivariance. Evaluating at $1_{C_2}$ shows that $\si=\si'$, and hence $\phi(f)=\phi(f')$. Thus, if $\phi$ is injective on primitive elements it is injective on all of $\oO$.

By \cref{treerotation}, all primitive elements of $\oO(n)$ are represented by a canonical tree $t\in \cantree{n}$, and thus $q\colon \cantree{n} \to \prim\oO(n)$ is surjective. It follows that the composite $\phi\circ q\colon \cantree{n} \to \prim \osubbe{C_2}(n)$ is also surjective.  Furthermore, for $t\in \cantree{n}$, the value of $\phi \circ q(t)$ on the generator of $C_2$ is the separable permutation produced by interpreting the leaves as $1_{\Sigma_1}$, $\otimes$ as block sum, and $\boxtimes$ as skew sum.  By the proof of \cite[Theorem 1]{ShapiroStephens}, separable permutations are in fact in bijection with canonical trees,\footnote{The reference \cite{ShapiroStephens} predates the term ``separable permutation,'' but the discussion on p.277 may be interpreted in this language.} and we conclude that $\phi \circ q$  is a bijection onto the primitive elements of $\osubbe{C_2}(n)$.  It follows that $\phi\colon \prim \oO(n) \to \prim \osubbe{C_2}(n)$ is also a bijection, as desired.
\end{proof}

As an aside, we note that it is possible to enumerate $\prim \osubbe{C_2}(n)$ in terms of the large Schr\"oder numbers.

\begin{defn}
The \emph{large Schr\"oder numbers} are the integers $S_i$ with $S_0=1$, $S_1=2$, and 
$S_n$ for $n\geq 2$ given by the recurrence relation
\[
  S_n = 3S_{n-1}+\sum_{k=1}^{n-2} S_k S_{n-k-1}.
\]
The first several terms in the sequence are
\[
    1, ~2, ~6, ~22, ~90, ~394, ~1,806, ~8,558, ~41,586, ~206,098, ~1,037,718,~\ldots.
\]
\end{defn}

By \cite[Theorem 1]{ShapiroStephens} and the bijection in our proof of \cref{thm:pres-c2}, we know that
\[
  |\prim \osubbe{C_2}(n)| = S_{n-1}.
\]
We initially discovered the connection between $\osubbe{C_2}$ and separable permutations via computer experimentation and reference to Sloane's OEIS \cite{OEIS}.

\subsection{Presentation for $Q_{C_3}$}\label{subsec:c3}

Now we consider $G=C_3$, and we will denote one of the generators by $g$. As it is the case for $C_2$, the operad $\osubbe{C_3}$ has three generators:
\begin{enumerate}
    \item $e \in \obe{C_3}(0)$,
    
    \item $\otimes \in \obe{C_3}(2)$, which is the function constant at the identity permutation, and
    
    \item $\boxtimes \in \obe{C_3}(3)$, which sends $1_{C_3}$ to the identity permutation $1_{\Si_3}$, $g$ to the permutation $(1\ 2\ 3)$, and $g^2$ to $(1\ 3\ 2)$.
\end{enumerate} 

Note that all primitive elements in $\obe{C_3}(2)$ are in $\osubbe{C_3}(2)$, one being $\otimes$ and the other three given by the partial compositions $\boxtimes \circ_i e$ for $i=1,2,3$ (see \cref{c3outputs}).

Similar to the case for $C_2$, we have a map of operads $\phi\colon \free{\{e,\otimes,\boxtimes\}} \rightarrow \osubbe{C_3}$ determined by sending each generator to its namesake. The main goal of this section is to prove the following theorem.

\begin{thm}\label{thm:pres-c3}
There is an isomorphism of operads $\free{\{e,\otimes,\boxtimes\}}/\langle R \rangle \cong \osubbe{C_3}$, where $R$ consists of the following:
\begin{enumerate}
    \item Reduction to identity: there is only one element in $\obe{C_3}(1)$, hence 
   \[\gamma(\boxtimes; \beid, e, e) = \gamma(\boxtimes;e,\beid,e) = \gamma(\boxtimes;e,e,\beid) = \gamma(\otimes; \beid,e) = \gamma(\otimes; e, \beid) = \beid;\]
    
    \item Strict associativity: primitive elements in $\osubbe{C_3}(2)$ follow strict associativity, i.e., for any $\dia \in \prim\osubbe{C_3}(2),$
    \[ \gamma(\dia; \dia, \beid) = \gamma(\dia; \beid, \dia);\] 
    
    \item Group action:  $g \cdot \boxtimes = \boxtimes \cdot (1\ 2\ 3)$, $g\cdot\otimes=\otimes$, $g\cdot e=e$.
  \end{enumerate}
\end{thm}

As in the case of $C_2$, the reader can check that all these relations are indeed satisfied in $\obe{C_3}$, and hence in $\osubbe{C_3}$. Thus, $\phi$ induces a surjective map of operads from $\oO=\free{\{e,\otimes,\boxtimes\}}/\langle R \rangle$ to $\osubbe{C_3}$. The hard work is to show that these relations generate all the rest, thus giving the desired isomorphism.

The proof is similar to that of \cref{thm:pres-c2}. We start by defining the trees that will be used in this context.

\begin{defn}\label{defn:treec3}
We now define $\tree{n}\subset \free{\{e,\otimes,\boxtimes\}}$ as the set of planar rooted trees with nodes labeled by $(1_G,e)$, $(1_G,\otimes)$ and $(1_G,\boxtimes)$, which we refer to by their second coordinate. Note that in this case $\boxtimes$ has three inputs and one output. Elements in $\oO$ represented by such trees are called \emph{primitive}. We call a tree \emph{reduced} if at any node the number of branches that are not marked by $e$ is at least 2. 

We call the collection of $\boxtimes$ and the four canonical binary trees, which are $\otimes$ and the three possible graftings of $e$ on $\boxtimes$, \emph{the essential nodes}. Their corresponding functions in $\osubbe{C_3}$ have outputs as in \cref{c3outputs}. Since their corresponding functions are distinct, we will not distinguish between these nodes and their associated function. 
\begin{center}
\begin{table}[h]
\begin{tabular}{|c|c|c|}
\hline
    essential primitive nodes & output of $g$ & output of $g^2$ \\ \hline 
    $\otimes$ & $\left (\begin{array}{cc}
        1 & 0 \\
        0 & 1
    \end{array} \right )$ & $\left (\begin{array}{cc}
        1 & 0 \\
        0 & 1
    \end{array} \right )$\\ \hline
    $\boxtimes$ & $\left (\begin{array}{ccc}
        0 & 0 & 1 \\
        1 & 0 & 0 \\
        0 & 1 & 0
    \end{array} \right )$ & $\left (\begin{array}{ccc}
        0 & 1 & 0 \\
        0 & 0 & 1 \\
        1 & 0 & 0
    \end{array} \right )$ \\ \hline
    $\gamma(\boxtimes; e, \beid, \beid)$ & $\left (\begin{array}{cc}
        0 & 1 \\
        1 & 0
    \end{array} \right )$ & $\left (\begin{array}{cc}
        1 & 0 \\
        0 & 1
    \end{array} \right )$\\ \hline
    $\gamma(\boxtimes; \beid, e, \beid)$ & $\left (\begin{array}{cc}
        0 & 1 \\
        1 & 0
    \end{array} \right )$ & $\left (\begin{array}{cc}
        0 & 1 \\
        1 & 0
    \end{array} \right )$\\ \hline
    $\gamma(\boxtimes; \beid, \beid, e)$ & $\left (\begin{array}{cc}
        1 & 0 \\
        0 & 1
    \end{array} \right )$ & $\left (\begin{array}{cc}
        0 & 1 \\
        1 & 0
    \end{array} \right )$\\ \hline
\end{tabular}
\caption{Outputs of essential primitive elements for $C_3$}\label{c3outputs}
\end{table}
\end{center}
 
 Lastly, the set $\cantree{n}$ of \emph{canonical trees} in this case consists of those trees that are reduced and such that for all essential binary nodes $\dia$, the node grafted on its right input, if there is any, is different than $\dia$. 
\end{defn}

\begin{proof}[Proof of \cref{thm:pres-c3}]
Following the strategy of the proof of \cref{thm:pres-c2}, we see that the the equivariance relation reduces the problem to proving that the primitive elements are in bijection.  By the analogue of \cref{treerotation}, this reduces to proving the following proposition.
\end{proof}

\begin{prop}\label{prop:c3cantree}
 The composite $\phi \circ q \colon \cantree{n} \to \prim\osubbe{C_3}(n)$ is a bijection for all $n\geq 0$.
\end{prop}

 We will show this by examining the matrices given by the canonical trees when evaluated at $g$ and $g^2$. 

\begin{defn}\label{uncoverednode}
Let $T\in \cantree{n}$. An \emph{uncovered node} of $T$ is a node such that no node of arity greater than 1 is grafted upon it, i.e., all its leaves are marked by $e$ or unmarked. 
\end{defn}

\begin{defn}\label{columnsimul}
Let $A, B$ be $n \times n$ permutation matrices and $C,D$ be $k \times k$ permutation matrices for some $k \leq n$. We say $C$ and $D$ occur \emph{$j$-column-simultaneously in $A$ and $B$, respectively} if $C$ and $D$ appear as blocks within $A$ and $B$, starting on the $j$-th column.\end{defn}

The following proposition describes how the uncovered nodes in a canonical tree $T$ are detected by the outputs of the corresponding function $\phi\circ q(T)$.

\begin{prop}[Uncovered Nodes]\label{cruciallem2}
Let $T\in \cantree{n}$, $f=\phi\circ q(T)$ be the primitive element in $\osubbe{C_3}(n)$ it represents, $t$ an essential node of arity $a$, and $j=1,\dots , n-a+1$. Then the following are equivalent:
\begin{enumerate}
\item The permutation matrices for $f(g)$ and $f(g^2)$ contain a $j$-column-simultaneous instance of $t(g)$ and $t(g^2)$, respectively, but not a $(j-1)$-column simultaneous instance.
\item There exists $T' \in \cantree{n-a+1}$ such that $T = T' \circ_j t$. 
\end{enumerate}
\end{prop}

\begin{proof}
We proceed by induction on $n$. For the base case $3$, one can derive the 29 possible canonical trees, and verify (1) and (2) are equivalent. 

Now, assume the statements are equivalent for all canonical trees of arity less than $n$, and let $T\in \cantree{n}$. Proving that (2) implies (1) for all five nodes is the easier direction, it follows from the definition of the operadic composition $\gamma$ and the fact that we are dealing with canonical trees. Indeed, if $T = T' \circ_j t$, operadic composition gives $f(g)$ and $f(g^2)$ containing a $j$-column-simultaneous instance of $t(g)$ and $t(g^2)$.
For $t =\boxtimes$, we cannot possibly have a $j-1$-column-simultaneous instance of $\boxtimes(g),\boxtimes(g^2)$ if we already have a $j$-column-simultaneous instance, since the result would be a pattern that cannot be a submatrix of a permutation matrix. This fact is illustrated below using $\boxtimes(g)$:
\[
\begin{tikzpicture}
	\matrix [matrix of math nodes] (m)
        		{
		0 & 0 & 1 &  \\ 
		1 & 0 & 0 & 1 \\ 
		0 & 1 & 0 & 0 \\ 
		   & 0 & 1 & 0  \\   
        		};   
        \draw[color = black] (m-1-1.north west) -- (m-1-3.north east) -- (m-3-3.south east)--(m-3-1.south west) -- (m-1-1.north west);
        \draw[color = black] (m-2-2.north west) -- (m-2-4.north east) -- (m-4-4.south east)--(m-4-2.south west) -- (m-2-2.north west);
        \draw[color = black] (m-2-1.west) -- (m-2-4.east);
\end{tikzpicture}
\qquad
\begin{tikzpicture}
	\matrix [matrix of math nodes] (m)
        		{
		    & 0 & 0 & 1 \\
		    & 1 & 0 & 0 \\ 
		 0 & 0 & 1 & 0 \\
		 1 & 0 & 0 &    \\
		 0 & 1 & 0 &    \\
        		};   
        \draw[color = black] (m-1-2.north west) -- (m-1-4.north east) -- (m-3-4.south east)--(m-3-2.south west) -- (m-1-2.north west);
        \draw[color = black] (m-3-1.north west) -- (m-3-3.north east) -- (m-5-3.south east)--(m-5-1.south west) -- (m-3-1.north west);
        \draw[color = black] (m-1-2.north) -- (m-5-2.south);
\end{tikzpicture}.\]

For $t = \otimes$, we use the inductive hypothesis. Note that in general, if $T = T' \circ_j t$, we can obtain the matrices corresponding to the values of the function for $T'$ by taking the columns and rows containing the $j$-column-simultaneous pattern and ``compressing'' them into a single row and column, where the entries in the pattern are compressed into a 1 and the entries outside of the pattern are compressed into a 0. The idea is shown below, where the pattern is encoded by the sub-matrix $A$: 
\[
\begin{matrix}
 && 0 &&  \\
&& \vdots  && \\
0 & \cdots & A & \cdots & 0\\
&& \vdots && \\
&& 0 && \\
\end{matrix} \qquad \to \qquad
\begin{matrix}
 && 0 &&  \\
&& \vdots  && \\
0 & \cdots & 1 & \cdots & 0.\\
&& \vdots && \\
&& 0 && \\
\end{matrix}\]

Suppose there is a $j-1$-column-simultaneous instance of $\otimes(g)$ and $\otimes(g)^2$ in $T=T'\circ_j t$. By the representation above, we see that  $\otimes(g)$ and $\otimes(g^2)$ occur $j-1$-column-simultaneously in $T'(g)$ and $T'(g^2)$. Since $T'$ itself is a canonical tree of lesser arity, we will use our inductive hypothesis to claim that the $\otimes$ we used to get to $T$ from $T'$ was grafted onto a ``tower of $\otimes$''. 

Indeed, if we follow the column-simultaneous instances of $\otimes(g)$ and $\otimes(g^2)$ as far left in the matrices for $T'$ as possible, say column $k\leq j-1$, then our inductive hypothesis tells us that $T'$ can be realized by grafting $\otimes$ on the $k$-th input of a canonical tree $T''$. If $k < j-1$, then by following the uncomposition, note the matrix of $T''$ has a $k$-column-simultaneous instance of $\otimes$, but not a $k-1$-column-simultaneous instance, but the $j-1$-column-simultaneous instance of $\otimes$ in $T'$ has now moved to column $j-2$. We can repeat the process until we have the original $j-1$-column-simultaneous instance of $T'$ in the $k$-th column, at which point we can return to $T'$ by regrafting $\otimes$ onto the right node of this $\otimes$. Namely, this tells us that $T = T' \circ_j \otimes$ contradicts $T$ being a canonical tree, since it must have been grafted onto the right node of $\otimes$ in $T'$. The reasoning is similar for the three other essential binary nodes. 

The heart of the inductive step is showing (1) implies (2), which is a five-in-one proof, one for each of the five essential primitive nodes. Again, we assume that the statements are equivalent for two canonical trees of arity less than n.


Let us begin with a canonical tree in $\cantree{n}$ representing some function $f$, and  assume that we have a $j$-simultaneous $\boxtimes$-pattern. There must be an uncovered node, $t'$ of arity $a$, attached to the $k$-th position of some canonical tree of lesser arity. 

If $k=j$ and $t'=\boxtimes$, we are done. If not, we remove $t'$ and call the new associated function $f'$. The matrices for $f'(g)$ and $f'(g^2)$ are obtained from the matrices of $f(g)$ and $f(g^2)$, respectively, by compressing the rows and columns containing the $k$-column-simultaneous patterns described above. The 1 in the compressed row and column corresponds to the unmarked leaf at the $k$-th entry of the new canonical tree. Now we want to show that not removing the entire $j$-column-simultaneous $\boxtimes$-pattern implies we have not removed any of the pattern. If this is the case, we have removed a node and arrived at a smaller arity canonical tree where we can use our induction hypothesis to say that the column-simultaneous $\boxtimes$-pattern indeed corresponds to an uncovered node, and that is not changed when we restore the removed node $t'$. We check for the impossibility of partial intersection for each of the five essential primitive nodes. 

We begin with $\boxtimes$ itself. Here it suffices to show that there is simply no way for $\boxtimes(g)$ to partially-intersect another $\boxtimes(g)$ without destroying the permutation matrix. Note that if the intersection contains a column or row with only 0s, then the pattern will contain two 1s in the same column or row, which means it can't be a subpattern of a permutation matrix. One can check that this is the case with the four possible intersections, two of which were shown earlier during the (2) implies (1) argument.
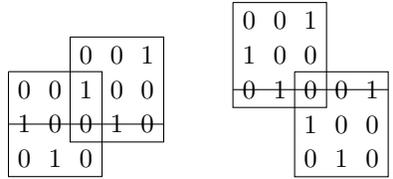
\[
\begin{tikzpicture}
	\matrix [matrix of math nodes] (m)
        		{
		    &    & 0 & 0 & 1 \\
		 0 & 0 & 1 & 0 & 0 \\ 
		 1 & 0 & 0 & 1 & 0 \\
		 0 & 1 & 0 &    &    \\
        		};   
        \draw[color = black] (m-2-1.north west) -- (m-2-3.north east) -- (m-4-3.south east)--(m-4-1.south west) -- (m-2-1.north west);
        \draw[color = black] (m-1-3.north west) -- (m-1-5.north east) -- (m-3-5.south east)--(m-3-3.south west) -- (m-1-3.north west);
        \draw[color = black] (m-3-1.west) -- (m-3-5.east);
\end{tikzpicture}
\qquad
\begin{tikzpicture}
	\matrix [matrix of math nodes] (m)
        		{
		     0 & 0 & 1 &    &    \\
		     1 & 0 & 0 &    &    \\
		     0 & 1 & 0 & 0 & 1 \\
		        &    & 1 & 0 & 0 \\
		        &    & 0 & 1 & 0  \\
        		}; 
        \draw[color = black] (m-1-1.north west) -- (m-1-3.north east) -- (m-3-3.south east)--(m-3-1.south west) -- (m-1-1.north west);  
        \draw[color = black] (m-3-3.north west) -- (m-3-5.north east) -- (m-5-5.south east)--(m-5-3.south west) -- (m-3-3.north west);
        \draw[color = black] (m-3-1.west) -- (m-3-5.east);
\end{tikzpicture}.
\]

Next is to check for the intersection of the $\boxtimes$-pattern with a primitive essential node $t'$ such that $t'(g) = 1_{\Sigma_2}$. There are six possible ways to intersect for $t'(g)$ and the $\boxtimes(g)$ pattern to intersect, with only one of them being feasible, the rest  having the same issue as above, that the intersection contains a row or column comprised of 0s:
\[
\begin{tikzpicture}
	\matrix [matrix of math nodes] (m)
        		{
			0 & 0 & 1 \\
			1 & 0 & 0 \\
			0 & 1 & 0 \\
		         		};   
        \draw[color = black](m-1-1.north west) -- (m-1-3.north east) -- (m-3-3.south east) -- (m-3-1.south west) -- (m-1-1.north west);
                \draw[color = black](m-2-1.north west) -- (m-2-2.north east) -- (m-3-2.south east) -- (m-3-1.south west) -- (m-2-1.north west);
\end{tikzpicture}.
\]
This case is a lost cause since we assume there is a column-simultaneous $\boxtimes(g)$ and $\boxtimes(g^2)$ pattern at that point. Regardless of $t'(g^2)$, it is an uncovered binary node, so we would need the first two columns of $\boxtimes(g^2)$ to form a $2 \times 2$ permutation matrix, but this is not the case, since $\boxtimes(g^2) = \begin{pmatrix}0 & 1 & 0 \\ 0 & 0 & 1 \\ 1 & 0 & 0\end{pmatrix}$.

Next, we assume $t'(g) = \begin{pmatrix} 0 & 1 \\ 1 & 0 \end{pmatrix}$. Note that this is not a submatrix in $\boxtimes(g)$, so in light of that, there are five possible overlaps, only one of which is feasible:
\[
\begin{tikzpicture}
	\matrix[matrix of math nodes] (m)
		{	    &    & 0 & 1 \\
			 0 & 0 & 1 & 0 \\
			1 & 0 & 0 &     \\
			0 & 1 & 0 &     \\
		};
	\draw(m-2-1.north west) -- (m-2-3.north east) -- (m-4-3.south east) -- (m-4-1.south west) -- (m-2-1.north west);
	\draw(m-1-3.north west) -- (m-1-4.north east)--(m-2-4.south east)-- (m-2-3.south west) -- (m-1-3.north west);
\end{tikzpicture}.
\]
 Again, this case runs into problems when considering $g^2$. No matter what $t'(g^2)$ is, we would break the permutation matrix for a similar intersection to occur on $\boxtimes(g^2)$:
\[
 \begin{tikzpicture}
	\matrix[matrix of math nodes] (m)
		{	 &&1&0\\
			 0 & 1 & 0 & 1 \\
			 0 & 0 & 1 &    \\
			 1 & 0 & 0 &     \\
		};
	\draw(m-2-1.north west) -- (m-2-3.north east) -- (m-4-3.south east) -- (m-4-1.south west) -- (m-2-1.north west);
	\draw(m-1-3.north west) -- (m-1-4.north east)--(m-2-4.south east)-- (m-2-3.south west) -- (m-1-3.north west);
	\draw(m-2-1.west) -- (m-2-4.east);
\end{tikzpicture}
\qquad
 \begin{tikzpicture}
	\matrix[matrix of math nodes] (m)
		{	 0 & 1 & 0 & 1 \\
			 0 & 0 & 1 & 0 \\
			1 & 0 & 0 &     \\
		};
	\draw(m-1-1.north west) -- (m-1-3.north east) -- (m-3-3.south east) -- (m-3-1.south west) -- (m-1-1.north west);
	\draw(m-1-3.north west) -- (m-1-4.north east)--(m-2-4.south east)-- (m-2-3.south west) -- (m-1-3.north west);
	\draw(m-1-1.west) -- (m-1-4.east);
\end{tikzpicture}
\qquad
\begin{tikzpicture}
	\matrix[matrix of math nodes] (m)
		{	0 & 1 & 0 &     \\
			0 & 0 & 1 &  0    \\
			1 & 0 & 0 & 1    \\
		};
	\draw(m-1-1.north west) -- (m-1-3.north east) -- (m-3-3.south east) -- (m-3-1.south west) -- (m-1-1.north west);
	\draw(m-2-3.north west) -- (m-2-4.north east)--(m-3-4.south east)-- (m-3-3.south west) -- (m-2-3.north west);
	\draw(m-3-1.west) -- (m-3-4.east);
\end{tikzpicture}
\qquad
\begin{tikzpicture}
	\matrix[matrix of math nodes] (m)
		{	0 & 1 & 0 &     \\
			0 & 0 & 1 &    \\
			1 & 0 & 0 & 1    \\
			& & 1&0\\
		};
	\draw(m-1-1.north west) -- (m-1-3.north east) -- (m-3-3.south east) -- (m-3-1.south west) -- (m-1-1.north west);
	\draw(m-3-3.north west) -- (m-3-4.north east)--(m-4-4.south east)-- (m-4-3.south west) -- (m-3-3.north west);
	\draw(m-3-1.west) -- (m-3-4.east);
\end{tikzpicture}.
\]

This covers the case for $\boxtimes$. We can do a similar argument for binary nodes. Let us begin with $t=\otimes$. 

 Like before, assuming the statements are equivalent for $\cantree{m}$ for $m < n$, we take a canonical tree in $\cantree{n}$, assume there is a $j$-column-simultaneous $\otimes$-pattern, and remove a top node $t'$ grafted in the $k$-th position. Assuming that the node $t'$ does not entirely intersect the $\otimes$-pattern, we must show that the node cannot intersect at all.

We need not consider the case of the removed node being a $\boxtimes$, since that was covered by the earlier case to not have a feasible intersection with $\otimes$. So instead, assume the removed node $t'$ satisfies $t'(g) = \begin{pmatrix} 0 & 1 \\ 1 & 0\end{pmatrix}$. It is not hard to check that there are no feasible intersections.

This argument generalizes to show the impossibility of any binary node intersecting with any binary node distinct from itself. So it suffices hereon to only discuss the case of self-intersection, which is taken care of by 
the assumption of canonical trees.

The two cases to consider are
\[
\begin{tikzpicture}
	\matrix[matrix of math nodes](m)
		{
			1 & 0 & \\
			0 & 1 & 0 \\
			   & 0 & 1 \\
		};
	\draw[dashed](m-1-1.north west) -- (m-1-2.north east) -- (m-2-2.south east) -- (m-2-1.south west) -- (m-1-1.north west);
	\draw(m-2-2.north west) -- (m-2-3.north east) -- (m-3-3.south east) -- (m-3-2.south west) -- (m-2-2.north west);
\end{tikzpicture}
\qquad
\begin{tikzpicture}
	\matrix[matrix of math nodes](m)
		{
			1 & 0 & \\
			0 & 1 & 0 \\
			   & 0 & 1 \\
		};
	\draw(m-1-1.north west) -- (m-1-2.north east) -- (m-2-2.south east) -- (m-2-1.south west) -- (m-1-1.north west);
	\draw[dashed](m-2-2.north west) -- (m-2-3.north east) -- (m-3-3.south east) -- (m-3-2.south west) -- (m-2-2.north west);
\end{tikzpicture}
\]
where the uncovered node $t'$ is denoted by the dashed line and the original $j$-column-simultaneous instanced is denoted by the solid line. Note that the first case cannot happen because of our assumption that there is no $(j-1)$-column simultaneous $\otimes$-pattern.

For the second case, note that after removing the uncovered node $t'$, which was grafted at $k=j+1$, we are left with $T' \in \cantree{n-1}$ which contains a $j$-column simultaneous instance of $t$. By the inductive hypothesis we have that $T'=T''\circ_j t$, and thus $T=(T''\circ_j t) \circ_{j+1} t'$, meaning that $T$ contains
\[
\begin{tikzpicture}[scale=0.75]
\begin{scope}[xshift=9pt,yshift=-5in,grow'=up,
frontier/.style={distance from root=150pt}]
\Tree [.$\otimes$ 
[.$\,$ ]
[.$\otimes$ 
[.$\,$ ] [.$\,$ ] ] ]
\end{scope}
\end{tikzpicture}
\]
as a subtree, implying that $T$ is not canonical. Thus, this case cannot happen either. This argument generalizes for the remaining binary nodes. 
\end{proof}

\begin{rem} The proof of \cref{cruciallem2} also tells us that there is no relation between $\boxtimes$ and itself. Since distinct graftings of $\boxtimes$ on itself yield distinct canonical trees, the proposition tells us that from the associated function, we can recover the canonical tree precisely by iteratively identifying uncovered nodes and decomposing the tree. Similarly, there are no relations between any two different essential primitive nodes. Also, any time we see a column-simultaneous pattern of size larger than 3, where the submatrices in the pattern are either the identity matrix or the antidiagonal, there are four essential binary nodes we can associate to these patterns (given by whether the submatrices for $g$ and $g^2$ are the identity or the antidiagonal). The leftmost $2\times 2$ block corresponds to that associated binary node being an uncovered node in the canonical tree corresponding to the pattern.

To illustrate the process of recovering the canonical tree corresponding to a primitive element, let us consider the following example: Suppose that $f \in \prim Q_{C_3}(7)$ with $f(g)$ and $f(g^2)$ corresponding the following two permutation matrices, respectively:
\[
\begin{tikzpicture}
\matrix[matrix of math nodes, left delimiter=(, right delimiter=)](m){
    0 & 0 & 0 & 0 & 0 & 0 & 1 \\
    0 & 0 & 0 & 0 & 1 & 0 & 0 \\
    0 & 0 & 0 & 0 & 0 & 1 & 0 \\
    1 & 0 & 0 & 0 & 0 & 0 & 0 \\
    0 & 1 & 0 & 0 & 0 & 0 & 0 \\
    0 & 0 & 1 & 0 & 0 & 0 & 0 \\
    0 & 0 & 0 & 1 & 0 & 0 & 0 \\
};
\draw (m-3-5.south west) rectangle (m-1-7.north east);
\draw (m-7-4.south west) rectangle (m-7-4.north east);
\draw (m-6-1.south west) rectangle (m-4-3.north east);
\end{tikzpicture},
\qquad
\begin{tikzpicture}
\matrix[matrix of math nodes, left delimiter=(, right delimiter=)](m){
    0 & 0 & 0 & 1 & 0 & 0 & 0 \\
    0 & 0 & 0 & 0 & 0 & 1 & 0 \\
    0 & 0 & 0 & 0 & 0 & 0 & 1 \\
    0 & 0 & 0 & 0 & 1 & 0 & 0 \\
    1 & 0 & 0 & 0 & 0 & 0 & 0 \\
    0 & 1 & 0 & 0 & 0 & 0 & 0 \\
    0 & 0 & 1 & 0 & 0 & 0 & 0 \\
};
\draw (m-4-5.south west) rectangle (m-2-7.north east);
\draw (m-1-4.south west) rectangle (m-1-4.north east);
\draw (m-7-1.south west) rectangle (m-5-3.north east);
\end{tikzpicture}.
\]
(The meaning of the boxes will be explained shortly.)  We want to find the canonical tree representation $T$ of $f$.

The reader might find \cref{c3outputs} helpful when following this process. From looking at the matrices, we see that the matrix corresponding to $\gamma(\otimes;\otimes,\beid)$ (identity matrix of size 3 for both $g$ and $g^2$) appear 1-column simultaneously in $f(g)$, $f(g^2)$. So the left-most uncovered node of $T$ is $\otimes$ and is grafted upon another $\otimes$.

Also, the matrices corresponding to $\boxtimes$ appear 5-column simultaneously in $f(g)$, $f(g^2)$. So  the right-most uncovered node of $T$ is $\boxtimes$. These two instances are blocked using solid lines in the matrices, and so is the remaining 1 in the fourth column. 

Now we replace the two column-simultaneous instances with 1, as described in the proof in \cref{cruciallem2}. One can see that the two new matrices are exactly the matrices corresponding to $\boxtimes$ for $g$ and $g^2$. Thus the node at the bottom of $T$ is labeled $\boxtimes$. In this way, we have obtained the canonical tree $T$ corresponding to $f$:

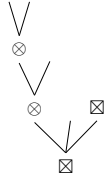
\[
\begin{tikzpicture}[scale=0.75]
\begin{scope}[xshift=9pt,yshift=-5in,grow'=up,
frontier/.style={distance from root=150pt}]
\Tree [.$\boxtimes$ [.$\otimes$ 
[.$\otimes$ 
[.$\,$ ] [.$\,$ ] ] [.$\,$ ] ]
[.$\,$ ]
[.$\boxtimes$ ] ]
\end{scope}
\end{tikzpicture}
\]
\end{rem}

We are now ready to complete the proof of \cref{thm:pres-c3} by proving \cref{prop:c3cantree}.

\begin{proof}[Proof of \cref{prop:c3cantree}]
Recall that $\phi \circ q$ is surjective. We prove that $\phi\circ q$ is injective by induction on the arity $n$ of the tree. The result is clear when $n=0,1$. If $n = 2$, there are only 4 elements in $\cantree{2}$, namely the binary essential primitive nodes, and they have different outputs at $g$ and $g^2$. As mentioned in the proof of \cref{cruciallem2}, one can check that the elements in $\cantree{3}$ represent different functions. Suppose this is true for all $m \leq n$, and consider the case $n+1$.

Let $T_1$, $T_2 \in \cantree{n+1}$ be such that they represent the same primitive function $f \in \osubbe{C_3}(n+1)$. Let $t$ be the left-most uncovered node in $T_1$. It corresponds to a $j$-column-simultaneous $t$-pattern in $f(g)$ and $f(g^2)$ satisfying condition (1) of \cref{cruciallem2}. Thus, we have that $T_1=T'_1\circ_j t$ and $T_2=T'_2\circ_j t$ for some canonical trees of lesser arity. Since the $T_1$ and $T_2$ represent the same function, the same is true for $T'_1$ and $T'_2$, since we are removing the same uncovered node. The inductive hypothesis tells us that $T_1'=T_2'$, and hence $T_1 = T_2$.\end{proof}

\section{Biased permutative equivariant categories for cyclic groups of order two and three}\label{sec:morphisms}

Now that we have explicit descriptions for the generators and the relations on $\osubbe{G}$ for $G=C_2$ and $C_3$, we turn to their categorical analogues and via \cref{thm:biased} give explicit biased descriptions of their algebras. 

We start with the definition of biased permutative equivariant categories, separating the cases of $C_2$ and $C_3$. For both, we denote by $g$ a chosen generator for the group.

\begin{defn}\label{dfn:biased-c2}
 A \emph{biased permutative $C_2$-category} consists of 
 \begin{itemize}
  \item a $C_2$-category $\cC$;
  \item a $C_2$-fixed object $e\in \cC$;
  \item a $C_2$-equivariant functor $\otimes \colon \cC \times \cC \to \cC$;
  \item a nonequivariant functor $\boxtimes \colon \cC \times \cC \to \cC$;
  \item a  $C_2$-natural isomorphism 
  \[
    \begin{tikzpicture}[x=1mm,y=1mm]
    \draw[tikzob,mm] 
    (0,0) node (00) {\cC \times \cC}
    (25,0) node (10) {\cC \times \cC}
    (12.5,-10) node (01) {\cC;}
    ;
    \path[tikzar,mm] 
    (00) edge node {\tau} (10)
    (10) edge node {\otimes} (01)
    (00) edge[swap] node {\otimes} (01)
    ;
    \draw[tikzob,mm]
    (12.5,-4) node {\Anglearrow{40} \beta}
    ;
  \end{tikzpicture}
  \]
   \item a nonequivariant natural isomorphism $\upsilon \colon \boxtimes \Rightarrow \otimes$ called the \emph{untwistor}, with components given by morphisms $\upsilon_{a,b}\colon a \boxtimes b \to a \otimes b$;
 \end{itemize}
 subject to the following axioms:
 \begin{enumerate}[(i)]
 \item $(\cC,\otimes,e,\beta)$ is a permutative category;
  \item $e$ is a strict two-sided unit for $\boxtimes$, that is, for all $a\in \cC$,
  \[e\boxtimes a = a = a \boxtimes e \qquad \text{and} \qquad \upsilon_{e,a}=\id_a=\upsilon_{a,e};\]
  \item $\boxtimes$ is strictly associative: for all $a,b,c\in \cC$,
  \[ a\boxtimes(b\boxtimes c) = (a \boxtimes b)\boxtimes c\]
  and the following diagram commutes
   \[
    \begin{tikzpicture}[x=30mm,y=15mm]
    \draw[tikzob,mm] 
    (0,0) node (0) {a\boxtimes (b \boxtimes c)}
    (0,-1) node (1) {a\otimes (b \boxtimes c)}
    (0,-2) node (2) {a\otimes (b \otimes c)}
    (1,0) node (3) {(a\boxtimes b) \boxtimes c}
    (1,-1) node (4) {(a\boxtimes b) \otimes c}
    (1,-2) node (5) {(a\otimes b) \otimes c;};
    \path[tikzar,mm] 
    (0) edge[swap] node { \upsilon_{a,b\boxtimes c}} (1)
    (1) edge[swap] node {\id_a \otimes \upsilon_{b,c}} (2)
    (3) edge node {\upsilon_{a\boxtimes b,c}} (4)
    (4) edge node {\upsilon_{a,b} \otimes \id_c} (5)
    (0) edge[/tikz/commutative diagrams/equal] (3)
    (2) edge[/tikz/commutative diagrams/equal] (5);
  \end{tikzpicture}
  \]
  \item for all $a,b\in \cC$
  \[g\cdot (a \boxtimes b) = (g\cdot b) \boxtimes (g\cdot a)\]
  and similarly for morphisms in $\cC$;
    \item for all $a,b \in \cC$, the following diagram commutes
    \[
    \begin{tikzpicture}[x=30mm,y=15mm]
    \draw[tikzob,mm] 
    (0,0) node (0) {g\cdot(a\boxtimes b)}
    (0,-2) node (1) {g\cdot(a\otimes b)}
    (1,0) node (2) {(g\cdot b)\boxtimes (g\cdot a)}
    (1,-1) node (3) {(g\cdot b)\otimes (g\cdot a)}
    (1,-2) node (4) {(g\cdot a)\otimes (g\cdot b).};
    \path[tikzar,mm] 
    (0) edge[swap] node {g\cdot \upsilon_{a,b}} (1)
    (2) edge node {\upsilon_{gb,ga}} (3)
    (3) edge node {\beta_{gb,ga}} (4)
    (0) edge[/tikz/commutative diagrams/equal] (2)
    (1) edge[/tikz/commutative diagrams/equal] (4);
  \end{tikzpicture}
  \]
 \end{enumerate}
\end{defn}

\begin{defn}\label{dfn:biased-c3}
 A \emph{biased permutative $C_3$-category} consists of 
 \begin{itemize}
  \item a $C_3$-category $\cC$;
  \item a $C_3$-fixed object $e\in \cC$;
  \item a $C_3$-equivariant functor $\otimes \colon \cC \times \cC \to \cC$;
  \item a functor $\boxtimes \colon \cC \times \cC \times \cC \to \cC$;
  \item a $C_3$-natural isomorphism 
  \[
    \begin{tikzpicture}[x=1mm,y=1mm]
    \draw[tikzob,mm] 
    (0,0) node (00) {\cC \times \cC}
    (25,0) node (10) {\cC \times \cC}
    (12.5,-10) node (01) {\cC;}
    ;
    \path[tikzar,mm] 
    (00) edge node {\tau} (10)
    (10) edge node {\otimes} (01)
    (00) edge[swap] node {\otimes} (01)
    ;
    \draw[tikzob,mm]
    (12.5,-4) node {\Anglearrow{40} \beta}
    ;
  \end{tikzpicture}
  \]
   \item a nonequivariant natural isomorphism $\upsilon$ called the \emph{untwistor}, whose components are given by morphisms $\upsilon_{a,b,c}\colon \boxtimes(a, b,c) \to (a \otimes b)\otimes c$.
    \end{itemize}
   To list the axioms, we note there are four (a priori distinct) binary operations: $\otimes$ and the three obtained from inserting $e$ in any of the positions of $\boxtimes$. There is also an associated instance of the untwistor for each of them. For example, if $a\dia b$  (temporarily) denotes $\boxtimes(e,a,b)$, we have
   \[\upsilon^\dia_{a,b}=\upsilon_{e,a,b}\colon a\dia b = \boxtimes(e,a,b) \longrightarrow (e\otimes a) \otimes b = a\otimes b.\]  The data above are subject to the following axioms
 \begin{enumerate}[(i)]
 \item $(\cC,\otimes,e,\beta)$ forms a permutative category;
  \item $e$ is a strict two-sided unit for all primitive binary operations, that is, for all $a\in \cC$ and all primitive binary operations $\dia$,
  \[e\dia a = a = a \dia e \qquad \text{and} \qquad \upsilon_{e,e,a}=\upsilon_{e,a,e}=\upsilon_{a,e,e}=\id_a;\]
  \item all primitive binary operations are strictly associative: for all $a,b,c\in \cC$ and all primitive binary operations $\dia$,
  \[a\dia(b\dia c) = (a \dia b)\dia c \]
  and the following diagram commutes
   \[
    \begin{tikzpicture}[x=30mm,y=15mm]
    \draw[tikzob,mm] 
    (0,0) node (0) {a\dia (b \dia c)}
    (0,-1) node (1) {a\otimes (b \dia c)}
    (0,-2) node (2) {a\otimes (b \otimes c)}
    (1,0) node (3) {(a\dia b) \dia c}
    (1,-1) node (4) {(a\dia b) \otimes c}
    (1,-2) node (5) {(a\otimes b) \otimes c;};
    \path[tikzar,mm] 
    (0) edge[swap] node { \upsilon^\dia_{a,b\dia c}} (1)
    (1) edge[swap] node {\id_a \otimes \upsilon^\dia_{b,c}} (2)
    (3) edge node {\upsilon^\dia_{a\dia b,c}} (4)
    (4) edge node {\upsilon^\dia_{a,b} \otimes \id_c} (5)
    (0) edge[/tikz/commutative diagrams/equal] (3)
    (2) edge[/tikz/commutative diagrams/equal] (5);
  \end{tikzpicture}
  \]
  \item for all $a,b,c\in \cC$
  \[g\cdot \boxtimes (a,b,c) =  \boxtimes (g\cdot c, g\cdot a, g\cdot b)\]
    and similarly for morphisms in $\cC$; 
    \item for all $a,b \in \cC$, the following diagram commutes
   \[
    \begin{tikzpicture}[x=40mm,y=15mm]
    \draw[tikzob,mm] 
    (0,0) node (0) {g\cdot \boxtimes (a,b,c)}
    (0,-2) node (1) {g\cdot(a\otimes b\otimes c)}
    (1,0) node (2) {\boxtimes (g\cdot c, g\cdot a, g\cdot b)}
    (1,-1) node (3) {(g\cdot c)\otimes (g\cdot a)\otimes (g\cdot b)}
    (1,-2) node (4) {(g\cdot a)\otimes (g\cdot b)\otimes (g\cdot c).};
    \path[tikzar,mm] 
    (0) edge[swap] node {g\cdot \upsilon_{a,b,c}} (1)
    (2) edge node {\upsilon_{gc,ga,gb}} (3)
    (3) edge node {\beta_{gc,ga\otimes gb}} (4)
    (0) edge[/tikz/commutative diagrams/equal] (2)
    (1) edge[/tikz/commutative diagrams/equal] (4);
  \end{tikzpicture}
  \]
 \end{enumerate}
\end{defn}

Note that in both definitions, $(\cC,\otimes, e,\beta)$ forms a naive permutative $G$-category, that is, a permutative category in which all the pieces are appropriately $G$-equivariant.

Recall from \cref{defn:QG} that $\subbe{G}$ is defined as the chaotic operad on $\osubbe{G}$, and thus, by the results of the previous section, we can describe $\subbe{G}$ as $\chao{\free{\{e,\otimes,\boxtimes\}}}/\langle R \rangle$, where $R$ is given by \cref{thm:pres-c2} for $G=C_2$ and by \cref{thm:pres-c3} for $G=C_3$.

\begin{thm}\label{thm:biased}
For $G=C_2$ and $C_3$, there is a one-to-one correspondence between biased permutative $G$-categories and algebras over $\subbe{G}$.
\end{thm}

\begin{proof}
 This follows from \cite[Theorem 2.10]{RubinThesis} and \cref{prop:chaoticQuotient}.  More precisely, consider the operad 
 \[\oO=\chao{\free{\{e,\otimes,\boxtimes\}}}/\langle g\cdot e = e, g\cdot \otimes = \otimes, g\cdot \boxtimes = \boxtimes \cdot \si \rangle,\]
 where $\si=(1\ 2)$ if $G=C_2$ and $(1\ 2 \ 3)$ if $G=C_3$. Note that $\oO$ is the operad
 $\mathcal{S}\mathcal{M}_\mathcal{N}$ of \cite[Definition 2.17]{RubinThesis} in the case where $\mathcal{N}$ contains the single $G$-set given by $G$ itself with action by left multiplication (and a choice of ordering). Thus, \cite[Theorem 2.10]{RubinThesis} implies that $\oO$-algebras correspond precisely to $\mathcal{N}$-normed symmetric monoidal categories (cf. \cite[Definition 2.3]{RubinThesis} for details). Note that this definition is very similar to ours, with the exception that there is an underlying symmetric monoidal structure (not necessarily strictly associative and unital), and that axiom (v) is required to hold for all elements of the group. 
 
 The relations from \cref{thm:pres-c2,thm:pres-c3} and \cref{prop:chaoticQuotient} imply that the underlying symmetric monoidal structure in our algebras will be strictly associative and unital, with the associator $\alpha$ and the unit constraints $\lambda$ and $\rho$ equal to the identity. Similarly, these results give the relations between the different instances of $\upsilon$. Requiring axiom (v) for a group generator ensures it holds for all elements of the group. This gives the desired result.
\end{proof}

\begin{rem}
 Analogously to \cite[Definition 2.6]{RubinThesis}, one can define appropriate notions of strict, pseudo and lax maps between biased permutative $G$-categories. These each constitute the morphisms of various 2-categories whose objects are biased permutative $G$-categories. The theorem above extends to give isomorphisms between these 2-categories and the respective ones of $\subbe{G}$-algebras (see \cref{rem:2cats}).
\end{rem}

\bibliographystyle{amsalpha}
\bibliography{eqOp}

\end{document}